\documentclass[a4paper]{amsart}

\usepackage{amssymb,amsmath,amsthm}
\usepackage{mathrsfs}
\usepackage{graphicx}
\usepackage{wrapfig}
\usepackage{caption}
\usepackage{subcaption}
\usepackage{enumerate}
\usepackage[all]{xypic}
 \usepackage{url}

\newcommand{\gt}{\textup{\textbf{GT}\thinspace}}
\newcommand{\gtf}{\textup{\textbf{GT$\mathcal{F}$}\thinspace}}

\title[Generalized topological spaces with associating function]{Generalized topological spaces \\ with associating function}
\author{Tomasz Witczak}

\address{Institute of Mathematics\\ University of Silesia\\ Bankowa~14\\ 40-007 Katowice\\ Poland}

\email{tm.witczak@gmail.com}

\date{}

\theoremstyle{Theorem}
\newtheorem{tw}{Theorem}[section]

\theoremstyle{Lemma}
\newtheorem{lem}[tw]{Lemma}

\theoremstyle{Remark}

\theoremstyle{Remark}

\theoremstyle{Definition}
\newtheorem{df}[tw]{Definition}
\newtheorem{przy}[tw]{Example}

\theoremstyle{Remark}

\usepackage{graphicx}

\begin{document}

\maketitle

\begin{abstract}
Generalized topological spaces in the sense of Cs\'{a}sz\'{a}r have two main features which distinguish them from typical topologies. First, these families of subsets are not closed under (finite) intersections. Second, the whole space may not be open. Hence, some points of the universe may be beyond any open set. In this paper we assume that such points are associated with certain open neighbourhoods by means of a special function $\mathcal{F}$. We study various properties of the structures obtained in this way. We introduce the notions of $\mathcal{F}$-interior and $\mathcal{F}$-closure and we discuss issues of convergence in this new setting. We pay attention to the fact that it is possible to treat our spaces as a semantical framework for modal logic. 
\end{abstract}

\section{Introduction}

Generalized topological spaces have been introduced by Cs\'{a}sz\'{a}r in the last years of twentieth century (see \cite{csaszar} and \cite{genopen}). They are investigated by many authors from all over the world. They have invented generalized analogues of separation axioms (see \cite{sep}, \cite{separ} and \cite{sarsak}), filters (see \cite{filters}), convergence (see \cite{baskar}, \cite{seq} and \cite{sarma}) or topological groups (see \cite{groups}). 

What is quite surprising, is the fact that Cs\'{a}sz\'{a}r's spaces are barely used in formal logic. Recently, we have prepared generalized topological semantics for certain weak modal logics (see \cite{witczak}). We have shown connections between our frames and certain subclasses of neighbourhood frames. Also we have prepared subintuitionistic version of our semantics. 

Our \emph{strong} generalized models (complete with respect to the modal logic $\mathbf{MTN4}$, as we can infer e.g. from \cite{indrze}) turned out to be similar to the \emph{complete extensional abstractions}, investigated by Soldano in \cite{soldano}. However, both his language and goals were different than ours. Moreover, he started from different intuitions. For us the primary goal was to check if it is possible to use Cs\'{a}sz\'{a}r's structures as a semantic tool for logic, while Soldano's approach was more practical (he wanted to model certain aspects of human reasoning and classifying objects in databases).

There is also an interesting paper by J\"{a}rvinen, Kondo and Kortelainen (see \cite{jarvi}). These autors used \emph{interior systems} (which are compatible with Cs\'{a}sz\'{a}r's generalized topologies and Soldano's extensional abstractions) to speak about soundness and completeness of logics $\mathbf{MT4}$ and $\mathbf{MTN4}$. Their approach was motivated by certain reflexions on the approximate reasoning and rough sets.

While working on this topic, we have developed certain purely topological notions and tools which seemed to be interesting. Basically, in Cs\'{a}sz\'{a}r's spaces it is possible that certain points are beyond any open set (and thus beyond maximal open set $\bigcup \mu$). We proposed an intermediate solution: that such points can have non-empty family of open neighbourhoods. Such family of sets is connected with these points by means of a special function $\mathcal{F}$. This approach allows to speak about new types of convergence and "openess" (this "openess" is weaker than the one which is typical for generalized topologies). 

\section{General overview of \gtf-spaces}

\subsection{Basic notions}

First of all, we repeat the very definition of \textit{generalized topological space} in the sense of Cs\'{a}sz\'{a}r (see \cite{csaszar} and \cite{genopen}). 

\begin{df}
Assume that $W$ is a non-empty set (universe) and $\mu \subseteq P(W)$. We say that $\mu$ is a generalized topology on $W$ if and only if: $\emptyset \in \mu$ and $\mu$ is closed under arbitrary unions, i.e. if $J \neq \emptyset$ and for each $i \in J$, $X_{i} \in \mu$, then $\bigcup_{i \in J} X_{i} \in \mu$. 

In such a case we say that $\langle W, \mu \rangle$ is a generalized topological space. The elements of $\mu$ are named $\mu$-open sets (or just open sets, if there is no risk of confusion) and for any $A \subseteq W$, we define $Int(A)$ as the union of all open sets contained in $A$.
\end{df}

Sometimes we shall say that all points from $W \setminus \bigcup \mu$ are \textit{orphaned}. As for the notion of \emph{closed set}, we say that the set $A \subseteq W$ is \emph{closed} if and only if its complement is open. We define $Cl(A)$ (closure of $A$) as the smallest closed set containing $A$. It is easy to show that $Cl(A) = W \setminus Int(W \setminus A)$ (see \cite{geno}).

The second thing to do is to establish our new structure, equipped with an additional function which connects orphaned points with open sets (open neighbourhoods). 

\begin{df}
\label{gtf}
We define $\gtf$-structure as a triple $M_\mu = \langle W, \mu, \mathcal{F} \rangle$ such that $\mu$ is a generalized topology on $W$ and $\mathcal{F}$ is a function from $W$ into $P(P(\bigcup \mu)$ such that:

\begin{itemize}
\item If $w \in \bigcup \mu$, then $[X \in \mathcal{F}_{w} \Leftrightarrow X \in \mu \text{ and } w \in X]$ {\normalfont [$\mathcal{F}_{w}$ is a shortcut for $\mathcal{F}(w)$]}
\item If $w \in W \setminus \bigcup \mu$, then $[X \in \mathcal{F}_{w} \Rightarrow X \in \mu]$
\end{itemize}
\end{df}




\begin{figure}[h]
\centering
\includegraphics[height=4.5cm]{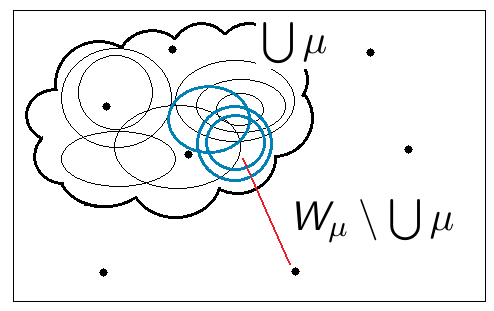}
\caption{Generalized topological model with function $\mathcal{F}$}
\label{fig:obrazek {mtop5}}
\end{figure}

Clearly, if $w \in \bigcup \mu$, then it belongs to \textit{each} of its open neighbourhoods. On the contrary, orphaned points do not belong to \textit{any} of their neighbourhoods. They are only \textit{associated} with them by means of $\mathcal{F}$. 

The next definition is just a useful shortcut:

\begin{df}
Assume that $\langle W, \mu, \mathcal{F} \rangle$ is a \gtf-structure and $A \in \mu$. Then we introduce the following notation: $A^{-1} = \{z \in W; A \in \mathcal{F}_{z}\}$.
\end{df}

Below we shall discuss simple example of \gtf-structure. Its basic form is based strictly on Ex. 3.1. from \cite{baskar}. 

\begin{przy}
\label{zet}
\normalfont{
Consider $\langle W, \mu, \mathcal{F} \rangle$, where $W = \mathbb{Z}$, $\mu = \{ \emptyset, \{1\}, \{1, 3\}, \{1, 3, 5\}, \{1, 3, 5, 7\}, ... \}$, $\mathcal{F}_{n} = \emptyset$ for any $n \in 2 \mathbb{Z}$ and if $n$ is odd, then $\mathcal{F}_{n}$ is just a collection of its open neighbourhoods.  

Of course, this is \gtf-structure, but undoubtedly it is rather degenerated case. However, we may think about more complex versions of this space. For example, we may replace $\mathcal{F}$ by:

\begin{enumerate}

\item $\mathcal{F}'$. Consider $\gamma: 2 \mathbb{Z} \rightarrow 2 \mathbb{Z} + 1$, where $\gamma(x) = \max \{m; m \in 2\mathbb{Z} + 1, m<x\}$. Assume that:

- if $n \in 2 \mathbb{Z} + 1$, then $G \in \mathcal{F}'_{m}$ $\Leftrightarrow$ $G \in \mu$ and $m \in G$. 

- if $n \in 2 \mathbb{Z}$, then $G \in \mathcal{F}'_{n}$ $\Leftrightarrow$ $G \in \mathcal{F}'_{\gamma(n)}$.

For example, $\mathcal{F}'_{8} = \mathcal{F}'_{\gamma(8)} = \mathcal{F}'_{7} = \{\{1, 3, 5, 7\}, \{1, 3, 5, 7, 9\}, \{1, 3, 5, 7, 9, 11\}, ... \}$.

\item $\mathcal{F}''$. It is just like $\mathcal{F}'$ but instead of $\gamma$ we use $\delta(x) = \min \{m; m \in 2 \mathbb{Z} + 1, m > x\}$. Then $\mathcal{F}''_{8} = \{ \{1, 3, 5, 7, 9\}, \{1, 3, 5, 7, 9, 11\}, ...\}$.

\item $\mathcal{F}'''$. We use $\gamma$ again and if $n \in 2 \mathbb{Z}$, then we define: $G \in \mathcal{F}'''_{n}$ $\Leftrightarrow$ $G \in \mu$, $G \neq \emptyset$ and $G \notin \mathcal{F}'''_{\gamma(n) + 2}$.

Now $\mathcal{F}'''_{8} = \{ \{1\}, \{1, 3\}, \{1, 3, 5\}, \{1, 3, 5, 7\}\}$. 

\end{enumerate}

Of course our associating function $\mathcal{F}$ can be very arbitrary but the most interesting cases are those with certain regularities or patterns. Later (when discussing convergence) we shall go back to the examples above. 
}
\end{przy}

In the next subsection we shall use our function $\mathcal{F}$ to establish some analogues of the well-known topological notions (like interior and closure).

\subsection{$\mathcal{F}$-interiors and $\mathcal{F}$-open sets}

The notions of $\mathcal{F}$-interior and $\mathcal{F}$-closure are based on the intuitions arising from the typical understanding of openness and closeness. However, one must note that there will be no full analogy. Our new concepts are somewhat weaker than their topological (and generalized topological) counterparts. This situation can be considered both as a limitation and a strength. 

\begin{df}
\label{fint}
Let $\langle W, \mu, \mathcal{F} \rangle$ be a \gtf and $w \in W$. Assume that $A \subseteq W$. We say that $w \in \mathcal{F}Int(A) \Leftrightarrow \text{ there is } G \in \mathcal{F}_{w} \text{ such that } G \subseteq A$.
\end{df}

In fact, we assume that $G \subseteq A \cap \bigcup \mu$. According to our earlier declarations, the definition above is modeled after the standard definition of interior in (generalized or not) topological spaces. Note, however, that in general we cannot say that $\mathcal{F}Int(A) \subseteq A$. To see this \footnote{Sometimes our examples and counter-examples will be presented only in the sketchy form.}, it is sufficient to consider any $A \in \mu$ and $w \in W \setminus \bigcup \mu$ such that $A \in \mathcal{F}_{w}$. Clearly, $w \in \mathcal{F}Int(A)$ but $w \notin A$. 

Now let us think about different situation: that $A \subseteq W \setminus \bigcup \mu$ and $w \in A$. Then $w \notin \mathcal{F}Int(A)$ (because even if there are some sets in $\mathcal{F}_{w}$, they cannot be contained in $A$, because they belong to $\mu$, while $A$ is beyond $\bigcup \mu$). This example (or rather sketch of reasoning) shows us that sometimes $A \nsubseteq \mathcal{F}Int(A)$. Of course, this lack of inclusion is not as surprising as the first one (in fact, it is normal also for ordinary and generalized interiors). Be as it may, the last example could be even more general: it is enough to assume that $A \cap (W \setminus \bigcup \mu) \neq \emptyset$, $\bigcup \mu \nsubseteq A$, $w \in A \cap (W \setminus \bigcup \mu)$ and for each $G \in \mathcal{F}_{w}$, $G \cap A = \emptyset$. 

For the reasons above, it is sensible to consider at least three concepts related to the notion of openness:

\begin{df}
Let $\langle W, \mu \rangle$ be a \gtf and $A \subseteq W$. We say that $A$ is:

\begin{itemize} 
\item $\mathcal{F}$-open ($\mathcal{F}$o.) iff $A = \mathcal{F}Int(A)$
\item d$\mathcal{F}$-open (d$\mathcal{F}$o.) iff $\mathcal{F}Int(A) \subseteq A$
\item u$\mathcal{F}$-open (u$\mathcal{F}$o.) iff $A \subseteq \mathcal{F}Int(A)$
\end{itemize}

\end{df}

Each of the following lemmas refers to the \gtf-structure $\langle W, \mu, \mathcal{F} \rangle$. Hence, we shall not repeat this assumption. 

\begin{lem}Assume that $A, B \subseteq W$ and $A \subseteq B$. Then $\mathcal{F}Int(A) \subseteq \mathcal{F}Int(B)$. \end{lem}

\begin{proof}This is obvious. If $v \in \mathcal{F}Int(A)$, then there is $G \in \mathcal{F}_{v}$ such that $G \subseteq A \subseteq B$.\end{proof}

\begin{lem}
\label{amu}
Assume that $A \subseteq W$. Then $Int(A) \subseteq \mathcal{F}Int(A)$. In particular, if $A \in \mu$, then $A \subseteq \mathcal{F}Int(A)$. \end{lem}

\begin{proof}If $Int(A) \neq \emptyset$ and $w \in Int(A)$, then there is certain $G \in \mu$ such that $w \in G \subseteq A$. Because $w \in \bigcup \mu$, $G \in \mathcal{F}_{w}$.

If $Int(A) = \emptyset$, then the result is obvious. \end{proof}

\begin{lem}
If $A \subseteq W$, then $\mathcal{F}Int(A) \cap \bigcup \mu = Int(A)$. 
\end{lem}

\begin{proof}
($\subseteq$) If $v \in \mathcal{F}Int(A) \cap \bigcup \mu$, then there is $G \in \mathcal{F}_{v}$ such that $G \subseteq A$. Of course, $v \in G$ (because $v \in \bigcup \mu$). Hence, $v \in Int(A)$.

($\supseteq$) If $v \in Int(A)$, then $v \in \mathcal{F}Int(A)$ (by means of Lemma \ref{amu}). But if $v \in Int(A)$, then $v \in \bigcup \mu$. Hence, $v \in \mathcal{F}Int(A) \cap \bigcup \mu$. 

\end{proof}

\begin{lem}If $A \subseteq W$, then $\mathcal{F}Int(\mathcal{F}Int(A)) \cap \bigcup \mu \subseteq \mathcal{F}Int(A)$. 
\end{lem} 

\begin{proof}
Assume that $v \in \mathcal{F}Int(\mathcal{F}Int(A)) \cap \bigcup \mu$. Then there is $G \in \mathcal{F}_{v}$ such that $G \subseteq \mathcal{F}Int(A)$. It means that for any $u \in G$ there is $H \in \mathcal{F}_{u}$ such that $H \subseteq A$. In particular, this is true for $v$ (because $v \in \bigcup \mu \supseteq G$). Thus $v \in G$ and $v \in \mathcal{F}Int(A)$. 

\end{proof}

\begin{lem}Assume that $A \in \mu$. Then $A^{-1} \subseteq \mathcal{F}Int(A)$. \end{lem}

\begin{proof}If $w \in A^{-1}$, then it means that $A \in \mathcal{F}_{w}$. Hence, there is $G \in \mathcal{F}_{w}$, namely $A$, such that $G \subseteq A$. Thus $w \in \mathcal{F}Int(A)$. \end{proof}

\begin{lem}$\mathcal{F}Int(W) = W \Leftrightarrow \text{ for any } w \in W, \mathcal{F}_{w} \neq \emptyset$. \end{lem}

\begin{proof}($\Rightarrow$) Assume that there is $v \in W$ such that $F_{v} = \emptyset$. Then $v \notin \mathcal{F}Int(W)$. ($\Leftarrow$) Assume that $v \notin \mathcal{F}Int(W)$. This means that for any $G \in \mathcal{F}_{v}$, $G \nsubseteq W$. Clearly, this is possible only if there are no any sets in $\mathcal{F}_{v}$. \end{proof}

\begin{lem}$\mathcal{F}Int(\emptyset) = \emptyset \Leftrightarrow \text{ for any } w \in W$, $\emptyset \notin \mathcal{F}_{w}$. \end{lem}

\begin{proof}
($\Rightarrow$) Assume that there is $v \in W$ such that $\emptyset \in \mathcal{F}_{v}$. Then $v \in \mathcal{F}Int(\emptyset)$. Contradiction.

($\Leftarrow$) Suppose that $\mathcal{F}Int(\emptyset) \neq \emptyset$. Then there must be at least one $v \in W$ for which there is $G \in \mathcal{F}_{v}$ such that $G \subseteq \emptyset$. Undoubtedly, $G$ must be empty.

\end{proof}

\begin{lem}Assume that for certain $X \subseteq W$ and for any $A \in \mu$: if $A \neq \emptyset$, then $A \nsubseteq X$. Then $\mathcal{F}Int(X) = \emptyset$ or $\mathcal{F}Int(X) \subseteq Z = \{z \in W; \emptyset \in \mathcal{F}_{z}\}$. \end{lem}

\begin{proof}Suppose that $\mathcal{F}Int(X) \neq \emptyset$ and $\mathcal{F}Int \nsubseteq Z$. Hence, there is $v \in \mathcal{F}Int(X)$ such that $\emptyset \notin \mathcal{F}_{v}$. It means that $F_{v}$ contains only non-empty sets. But we assumed that there are no non-empty sets (from $\mu$) contained in $X$. \end{proof}

Note that if there is at least one $X \subseteq W$ such that $\mathcal{F}Int(X) = \emptyset$, then $Z$ (defined as above) must be empty. Assume the contrary. If $z \in Z$, then we can always say that there is $G \in F_{z}$, namely $G = \emptyset$, such that $G \subseteq X$. But then $z \in \mathcal{F}Int(X)$. 



\begin{lem}
\label{sumint}
Suppose that $J \neq \emptyset$ and $\{X_i\}_{i \in J}$ is a family of subsets of $W$. Then $\bigcup_{i \in J} \mathcal{F}Int(X_i) \subseteq \mathcal{F}Int(\bigcup_{i \in J} X_i)$. If each $X_i$ is u$\mathcal{F}$o., then $\bigcup_{i \in J} X_{i} \subseteq \mathcal{F}Int(\bigcup_{i \in J} X_i)$. 
\end{lem}

\begin{proof}
Let $v \in \bigcup_{i \in J}\mathcal{F}Int(X_i)$. Hence, there is $k \in J$ such that $v \in \mathcal{F}Int(X_{k})$. Then there is $G \in \mathcal{F}_{v}$ such that $G \subseteq X_k$. But then $G \subseteq X_k \subseteq \bigcup_{i \in J} X_i$. Hence, we can say that $v \in \mathcal{F}Int(\bigcup_{i \in J} X_i)$. 
\end{proof}

Note that we can easily imagine the following situation: there is $v \in W$ such that for any $G \in \mathcal{F}_{v}$ and for each $i \in J$, $G \nsubseteq \mathcal{F}Int(X_i)$ but at the same time there is $H \in \mathcal{F}_{v}$ such that $H \subseteq \bigcup_{i \in J} \mathcal{F}Int(X_i)$. Hence, $v \in \mathcal{F}Int(\bigcup_{i \in J}X_i)$ but $v \notin \bigcup_{i \in J} \mathcal{F}Int(X_i)$. Please take a look below:

\begin{przy}
\normalfont{
Let $\langle W, \mu, \mathcal{F} \rangle$ be \gtf where $W = \mathbb{R}^2$ and $\mu$ is a standard (hence, in particular, generalized) topology but with respect to the ball $K[(0,0), 10]$. It means that $\mu$-open sets are unions of open balls contained in $K$, i.e $K = \bigcup \mu$. Now we can take $K_1[(0,0), 2]$, $K_2[(2, 0), 2]$, $v = (20, 20)$ and assume that $\mathcal{F}_{v}$ contains only one set, namely open rectangle based on coordinates $(-3, -1), (-3, 1), (3, 1), (3, -1)$. Clearly, it is contained in $K_1 \cup K_2$ (and in the union of their $\mathcal{F}$-interiors, because these balls are open) but is not contained in $K_1$ nor $K_2$. 
}
\end{przy}

\begin{lem}
Suppose that $J \neq \emptyset$ and $\{X_i\}_{i \in J}$ is a family of subsets of $W_\mu$. Then $\mathcal{F}Int(\bigcap_{i \in J} X_i) \subseteq \bigcap_{i \in J} \mathcal{F}Int(X_i)$. If each $X_i$ is d$\mathcal{F}$o., then $\mathcal{F}Int(\bigcap_{i \in J} X_i) \subseteq \bigcap_{i \in J} X_i$. 
\end{lem}

\begin{proof}The proof is similar to the former one. Also we can find counterexample for the opposite inclusion.\end{proof}

\begin{lem}
\label{uopen}
Assume that $A \subseteq \mathcal{F}Int(A) \subseteq \bigcup \mu$. Then $A \in \mu$.
\end{lem}

\begin{proof}
For any $v \in A$, there is $G \in \mathcal{F}_{v}$ such that $G \subseteq A$. Of course, $G \in \mu$. But $v \in \bigcup \mu$, hence $v \in G$ (from the very definition of $\mathcal{F}$ in \gtf-structure). This means that for any $v \in A$, $v \in Int(A)$. Thus $A \subseteq Int(A)$, so $A$ is open, i.e. $A \in \mu$. 
\end{proof}

As we said, $\mathcal{F}$-open sets do not have all properties of open sets, even in Cs\'{a}sz\'{a}r's sense. Nonetheless, the following theorem should be considered as useful. In fact, it \textit{will} be usefull in our further investigations. 

\begin{tw}
\label{eachopen}
Let $\langle W, \mu, \mathcal{F} \rangle$ be a \gtf-structure and $w \in W$. Then $\mathcal{F}_{w} \neq \emptyset$ $\Leftrightarrow$ there is $\mathcal{F}$o. set $G \subseteq W$ such that $w \in G$. 
\end{tw}

\begin{proof}

($\Rightarrow$)

Since $\mathcal{F}_{w} \neq \emptyset$, then there is at least one $A \in \mathcal{F}_{w}$. Of course, $w \in \mathcal{F}Int(A)$. If $A = \mathcal{F}Int(A)$, then we can finish our proof. If not, then it means that $\mathcal{F}Int(A) \nsubseteq A$. Let us define $G$ as $A \cup \mathcal{F}Int(A)$. We show that $G$ is open, i.e. that $\mathcal{F}Int(G) = G$. 

($\subseteq$) Let $v \in \mathcal{F}Int(G)$. Hence, $v \in \mathcal{F}Int(A \cup \mathcal{F}Int(A))$. Now there is $U \in \mathcal{F}_{v}$ such that $U \subseteq A \cup \mathcal{F}Int(A)$. In fact, it means that $U \subseteq A$ (because $U \subseteq \bigcup \mu$ and $\mathcal{F}Int(A) \cap \bigcup \mu = Int(A) = A$). Thus $v \in \mathcal{F}Int(A)$. From this we infer that $v \in G$. 

($\supseteq$) Let $v \in G$. Hence, $v \in A$ or $v \in \mathcal{F}Int(A)$. If $v \in A$, then $A \in \mathcal{F}_{v}$ (because $A \in \mu$). Therefore $v \in \mathcal{F}Int(G)$. If $v \in \mathcal{F}Int(A)$, then there is $U \in \mathcal{F}_{v}$ such that $U \subseteq A \subseteq G$. Hence, $v \in \mathcal{F}Int(G)$. 

($\Leftarrow$)

Suppose that $G \subseteq W$ is $\mathcal{F}o.$, $w \in G$ and $\mathcal{F}_{w} = \emptyset$. Of course $\mathcal{F}Int(G) = G$, so $w \in \mathcal{F}Int(G)$. Hence, there is $H \in \mathcal{F}_{w}$ such that $H \subseteq G$. Contradiction. 

\end{proof}

Note that in the case of $\Leftarrow$-direction it is enough to assume that $G$ is $u\mathcal{F}$o.

\subsection{$\mathcal{F}$-closures and $\mathcal{F}$-closed sets}

Any sensible definition of "openess" should be dual to certain understanding of "closeness". We propose the following notion, based on the very well known property of closed (and generalized closed) sets. 

\begin{df}
\label{fclosur}
Let $\langle W, \mu, \mathcal{F} \rangle$ be a \gtf and $w \in W$. Assume that $A \subseteq W$. We say that $w \in \mathcal{F}Cl(A) \Leftrightarrow \text{ for any } G \in \mathcal{F}_{w}, G \cap A \neq \emptyset$.
\end{df}

Now we can define $\mathcal{F}$-closed sets:

\begin{df}
Let $\langle W, \mu, \mathcal{F} \rangle$ be a \gtf and $A \subseteq W$. We say that:

\begin{itemize} 
\item $A$ is $\mathcal{F}$-closed ($\mathcal{F}$c.) if and only if $\mathcal{F}Cl(A) = A$. 
\item d$\mathcal{F}$-closed (d$\mathcal{F}$c.) iff $\mathcal{F}Cl(A) \subseteq A$
\item u$\mathcal{F}$-closed (u$\mathcal{F}$c.) iff $A \subseteq \mathcal{F}Cl(A)$
\end{itemize}
\end{df}

This definition makes sense because it gives us expected dualism:

\begin{tw}
Let $\langle W, \mu, \mathcal{F} \rangle$ be a \gtf. Assume that $A \subseteq W$ is $\mathcal{F}$-open. Then $-A$ is $\mathcal{F}$-closed. 
\end{tw}

\begin{proof}
We know that $\mathcal{F}Int(A) = \{z \in W; \text{ there is } G \in \mathcal{F}_{z} \text{ such that } G \subseteq A\} = A$. Let us consider $-A = \{z \in W; \text{ for each } G \in \mathcal{F}_{z}, G \nsubseteq A\}$. We shall show that $\mathcal{F}Cl(-A) = -A$. 

($\subseteq$) Assume that $w \in \mathcal{F}Cl(-A)$. Hence, for any $G \in \mathcal{F}_{w}$, $G \cap -A \neq \emptyset$. Now $G \nsubseteq A$ and for this reason $w \in -A$. 

($\supseteq$) Suppose that $w \in -A$ and assume that there is $H \in \mathcal{F}_{w}$ such that $H \cap -A = \emptyset$. It means that $H \subseteq A$. But then $w \in \mathcal{F}Int(A) = A$ which gives us plain contradiction. 
\end{proof}

As in the case of interiors, properties of $\mathcal{F}Cl$ are rather weak (when compared to the properties of closures and generalized closures). For example, we may ask if $A \subseteq \mathcal{F}Cl(A)$. The answer is (in general) negative. We may easily imagine the following situation: $A \subseteq W$, $A \cap (W \setminus \bigcup \mu) \neq \emptyset$, $w \in A \cap (W \setminus \bigcup \mu)$ and there is $G \in \mathcal{F}_{w}$ such that $G \cap A = \emptyset$. 

On the other hand, it is possible that $\mathcal{F}Cl(A) \nsubseteq A$. It is even easier to imagine that for any $G \in \mathcal{F}_{w}$, $G \cap A \neq \emptyset$ but at the same time $w \notin A$ (we may assume that $w \in W \setminus \bigcup \mu$ and $A \subseteq \bigcup \mu$). 

We have the following lemmas (with respect to an arbitrary $\langle W, \mu, \mathcal{F} \rangle$):

\begin{lem}
$\mathcal{F}Cl(\emptyset) = \emptyset \Leftrightarrow \text{ for any } z \in W, \mathcal{F}_{z} \neq \emptyset$
\end{lem}

\begin{proof}
Note that $\mathcal{F}Cl(\emptyset) = \{z \in W; G \cap \emptyset \neq \emptyset \text{ for each } G \in \mathcal{F}_{z}\}$. 

($\Rightarrow$) If we assume that there is $z \in W$ such that $\mathcal{F}_{z} = \emptyset$, then we can say anything about an arbirtrary "$G \in \mathcal{F}_{z}$". In particular, we can say that "such $G$" has non-empty intersection with $\emptyset$. Hence, $z \in \mathcal{F}Cl(\emptyset)$ and $\mathcal{F}Cl(\emptyset) \neq \emptyset$.

($\Leftarrow$) If $\mathcal{F}Cl(\emptyset)$ is not empty, then there must be at least one $z \in \mathcal{F}Cl(\emptyset)$. Now if we assume that $\mathcal{F}_{z} \neq \emptyset$ (i.e. there is certain $G \in \mathcal{F}_{z}$), then it means that $\emptyset$ forms non-empty intersection with $G$. This is not possible.

\end{proof}

\begin{lem}
$\mathcal{F}Cl(W) = W \Leftrightarrow \text{ for any } z \in W, \emptyset \notin \mathcal{F}_{z}$. 
\end{lem}

\begin{proof}
Note that $\mathcal{F}Cl(W) = \{z \in W; G \cap W \neq \emptyset \text{ for each } G \in \mathcal{F}_{z}\}$. 

($\Rightarrow$) If the set in question set is equal to $W$, then it is impossible that there exists $z \in W$ such that $\emptyset \in \mathcal{F}_{z}$. It would mean that $\emptyset \cap W \neq \emptyset$. 

($\Leftarrow$) On the other hand, if there is $z \in W$ such that $\emptyset \in \mathcal{F}_{z}$, then of course $\emptyset \cap W = \emptyset$. Hence, $z \notin \mathcal{F}Cl(W)$. 
\end{proof}

\begin{lem}
Suppose that $J \neq \emptyset$ and $\{X_i\}_{i \in J}$ is a family of subsets of $W$. Then $\bigcup_{i \in J} \mathcal{F}Cl(X_i) \subseteq \mathcal{F}Cl(\bigcup_{i \in J} X_i)$. If each $X_i$ is u$\mathcal{F}$c., then $\bigcup_{i \in J} X_{i} \subseteq \mathcal{F}Cl(\bigcup_{i \in J} X_i)$. 
\end{lem}

\begin{proof}
Let $v \in \bigcup_{i \in J} \mathcal{F}Cl(A)$. It means that there is $k \in J$ such that $v \in \mathcal{F}Cl(X_{k})$. Hence, for any $G \in \mathcal{F}_{v}$, $G \cap X_{k} \neq \emptyset$. Clearly, $X_{k} \subseteq \bigcup_{i \in J} X_i$. Thus $G \cap \bigcup_{i \in J} X_i \neq \emptyset$. For this reason, $v \in \mathcal{F}Cl(\bigcup_{i \in J} X_i)$. 
\end{proof}

As in the case of $\mathcal{F}$-interiors, we can find counter-example for the opposite inclusion.

\begin{przy}
\normalfont{
Let us go back to the \gtf used after the Lemma \ref{sumint} but now assume that $\mathcal{F}_{v}$ contains only two sets, namely open balls: $L_1[(-1, 0), 0.5]$ and $L_2[(3, 0), 0.5]$. Clearly, each of them has non-empty intersection with $K_1 \cup K_2$ (in fact, they are contained in this union), hence $v \in \mathcal{F}Cl(K_1 \cup K_2)$. On the other hand, $v$ is not in $\mathcal{F}Cl(K_1)$ (because $L_2 \cap K_1 = \emptyset$) and is not in $\mathcal{F}Cl(K_2)$ (because $L_1 \cap K_2 = \emptyset$).
}
\end{przy}

The next lemma is about intersections:

\begin{lem}
Suppose that $J \neq \emptyset$ and $\{X_i\}_{i \in J}$ is a family of subsets of $W_\mu$. Then $\mathcal{F}Cl(\bigcap_{i \in J} X_i) \subseteq \bigcap_{i \in J} \mathcal{F}Cl(X_i)$. If each $X_i$ is d$\mathcal{F}$c., then $\mathcal{F}Cl(\bigcap_{i \in J} X_i) \subseteq \bigcap_{i \in J} X_i$. 
\end{lem}

\begin{proof}
Let $v \in \mathcal{F}Cl(\bigcap_{i \in J}X_i)$. It means that for any $G \in \mathcal{F}_{v}$, $G \cap \bigcap_{i \in J} X_i \neq \emptyset$. Hence, for any $k \in J$, $G \cap X_k \neq \emptyset$. Thus $v \in \mathcal{F}Cl(X_k)$. Finally, $v \in \bigcap_{i \in J} \mathcal{F}Cl(X_i)$. 
\end{proof}

As earlier, the converse is not true (in general). Finally, we should prove quite simple but important lemma:

\begin{lem}
\label{close}
Assume that $\langle W, \mu, \mathcal{F} \rangle$ is a \gt-frame and $A \subseteq W$. Then $\mathcal{F}Cl(A) \subseteq Cl(A)$. 
\end{lem}

\begin{proof}
Let us assume that $w \in \mathcal{F}Cl(A)$. Hence, for any $G \in \mathcal{F}_{w}$, $G \cap A \neq \emptyset$. If $w \in \bigcup \mu$, then the result is obvious: each $G \in \mathcal{F}_{w}$ is just an ordinary generalized topological neighbourhood. Hence, $w \in Cl(A)$. If $w \in W \setminus \bigcup \mu$, then there are no any ordinary neighbourhoods of $w$. Thus, our conclusion is trivially true. 
\end{proof}

\begin{figure}[h]
\centering
\includegraphics[height=5.5cm]{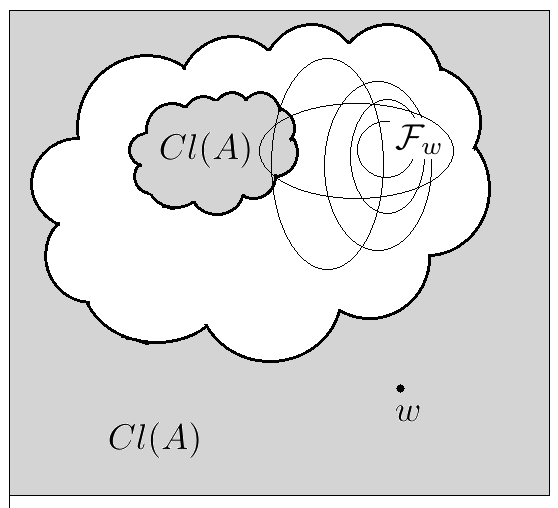}
\caption{$w \notin \mathcal{F}Cl(A)$ but $w \in Cl(A)$}
\label{fig:obrazek {asso1}}
\end{figure}

Of course, the converse is not true. We shall not present detailed counterexample, just a scheme based on a Fig. 1. Clearly, $w \in W \setminus Int(W \setminus A) = Cl(A)$ but at the same time there are some sets in $\mathcal{F}_{w}$ which have empty intersection with $A$. 

\section{Generalized nets and sequences}

In this section we adhere mostly to the notions introduced and investigated by Baskaran in \cite{baskar} and \cite{seq}. Of course, they are placed in our specific environment. Moreover, we have developed some new definitions and ideas. We have been inspired also by Dev Sarma (see \cite{sarma}) and Palaniammal, Murugalingam (see \cite{filters}). 

In the presence of $\mathcal{F}$ we can introduce various definitions of convergence, limit and limit point. The first definition refers to the notion of generalized net and is based on the Baskaran's one:

\begin{df}
Let $\langle W, \mu, \mathcal{F} \rangle$ be a \gtf and $(P, \geq)$ be a poset. A generalized net (gnet) in $W$ is a function $f: P \rightarrow W$. The image of $\lambda \in P$ under $f$ is denoted by $f_\lambda$ and the whole gnet is denoted as $(f_\lambda)$.
\end{df}

When there is no risk of confusion (e.g. when $P$ may be arbitrary or when we are working only with one gnet), we shall not always write directly that $f$ is a function from $P$ into $W$. It will be known from the context that $P$ in question is connected to the given $(f_\lambda)$ and \emph{vice-versa}.

What may be surprising, is the fact that generalized net has \textit{pre-ordered} domain and not necessarily directed (this was assumed both by \cite{baskar} and \cite{sarma}). For this reason we can introduce also two other notions:

\begin{df}
Let $\langle W, \mu, \mathcal{F} \rangle$ be a \gtf and $(P, \geq)$ be a poset. We say that gnet $(f_\lambda)$, $f: P \rightarrow W$, is net if and only if $P$ is directed, i.e. for any two elements $\lambda_1, \lambda_2 \in P$ there is $\lambda_3 \in P$ such that $\lambda_1 \leq \lambda_3$ and $\lambda_2 \leq \lambda_3$. If $P = \mathbb{N}$, then we say that $(f_\lambda)$ is a sequence.
\end{df}

Now we go to the convergence, using $\mathcal{F}$ directly:

\begin{df}
Let $\langle W, \mu, \mathcal{F} \rangle$ be a \gtf and $(f_\lambda)$ be a gnet in $W$. We say that:

\begin{itemize}
\item $(f_\lambda)$ is eventually in $U \subseteq W$ iff there is $\lambda_0 \in P$ such that for any $\lambda \geq \lambda_0$, $f_\lambda \in U$. 

\item $(f_\lambda)$ converges to $w \in W$ (i.e. $(f_\lambda) \rightarrow w$) iff for any $G \in \mathcal{F}_{w}$, $f_\lambda$ is eventually in $G$. In this case we say that $w$ is a limit of $(f_\lambda)$. We say that $(f_\lambda)$ is convergent if there is $v \in W$ such that $(f_\lambda) \rightarrow v$. 

\item $(f_\lambda)$ is frequently in $U$ iff for any $\lambda \in P$ there is $\lambda_1 \in P$ such that $\lambda_1 \geq \lambda$, we have $f_{\lambda_1} \in U$. We say that $w$ is a limit point of $(f_\lambda)$ if it is frequently in every $G \in \mathcal{F}_{w}$. 
\end{itemize}
\end{df}

Sometimes, when it is useful for clarity, we shall name this kind of convergence as $\mathcal{F}$-convergence or $\rightarrow$-convergence. Contrary to the result for \gt (without $\mathcal{F}$), in our environment constant gnet may not be convergent. Let us formulate the following lemma and theorem:

\begin{lem}
Let $\langle W, \mu, \mathcal{F} \rangle$ be a \gtf and $(f_\lambda) = (w)$ be a constant gnet in $W$. Then $(w)$ is convergent $\Leftrightarrow$ $(w) \rightarrow w$. 
\end{lem}

\begin{proof}
($\Leftarrow$) This is obvious.

($\Rightarrow$) Assume the contrary, i.e. that there is $v \in W, v \neq w$ such that $(w) \rightarrow v$ but $(w) \nrightarrow w$. Hence, for any $G \in \mathcal{F}_{v}$, $w \in G$ (note that we speak about \emph{constant} gnet) but there still is $H \in \mathcal{F}_{w}$ such that $w \notin H$. But if $w$ is in each open neighbourhood of $v$, then $w$ must be in $\bigcup \mu$. Then for any $G \in \mathcal{F}_{w}$, $w \in G$, hence the existence of $H$ is not possible.

\end{proof}

\begin{tw}
Let $\langle W, \mu, \mathcal{F} \rangle$ be a \gtf and $(f_\lambda) = (w)$ be a constant gnet in $W$. Then $(f_\lambda)$ is convergent $\Leftrightarrow$ $w \in \bigcup \mu$ or $\mathcal{F}_{w} = \emptyset$. 
\end{tw}

\begin{proof}
($\Leftarrow$) Assume that $(f_\lambda)$ is not convergent. In particular (by the preceeding lemma) it means that $(w) \nrightarrow w$. Hence, there is $G \in \mathcal{F}_{w}$ such that $w \notin G$. Now we have two options. If $w \in \bigcup \mu$, then $w \in G$, this is contradiction. If $\mathcal{F}_{w} = \emptyset$ (which means, in particular, that $w \in W \setminus \bigcup \mu$), then $w \notin G \subseteq \bigcup \mu$. 

($\Rightarrow$) Now $(w)$ is convergent. In particular, it means that $(w) \rightarrow w$. Suppose that $w \notin \bigcup \mu$ and $\mathcal{F}_{w} \neq \emptyset$. But then for any $G \in \mathcal{F}_{w}$, $w \notin G$. Hence, $(f_\lambda) = (w)$ is not eventually in $G$. Contradiction with convergence. 
\end{proof}

The next question about constant gnets is: is the limit of convergent gnet unique? Let us introduce certain subclass of our structures. 

\begin{df}
We say that \gtf-structure $\langle W, \mu, \mathcal{F} \rangle$ is $\mathcal{F}T_1$ $\Leftrightarrow$ for any $w \neq v$ there are $G \in \mathcal{F}_{w}$ such that $v \notin G$ and $H \in \mathcal{F}_{v}$ such that $w \notin H$. 
\end{df}

\begin{tw}
Let $\langle W, \mu, \mathcal{F} \rangle$ be a \gtf-structure. Then the limit of every constant and convergent gnet is unique $\Leftrightarrow$ $\langle W, \mu, \mathcal{F} \rangle$ is $\mathcal{F}T_1$. 
\end{tw}

\begin{proof}
($\Rightarrow$) Assume that $(f_\lambda)$ has unique limit $w$. Hence, for any $v \neq w$, $f_\lambda = (w) \nrightarrow v$. Thus there is $H \in \mathcal{F}_{v}$ such that $w \notin H$. 

But maybe for any $G \in \mathcal{F}_{w}$, $v \in G$? This would mean that $(v) \rightarrow w$ (by the very definition of convergence). But $(v) \rightarrow v$ and the limits are unique, so $v = w$. Contradiction.

($\Leftarrow$) Suppose that our space is $\mathcal{F}T_1$. Let $w \neq v$, and $(w)$ be a convergent gnet. Then $(w) \rightarrow w$. Assume that at the same time $(w) \rightarrow v$. It means that for any $G \in \mathcal{F}_{v}$, $w \in G$. But this is contradiction.

\end{proof}

The following theorem is nearly compatible with ($\Rightarrow$) part of Th. 13 in \cite{baskar}. However, we must assume that our gnet $(w)$ is convergent. 

\begin{tw}
Let $\langle W, \mu, \mathcal{F} \rangle$ be a \gtf. Assume that $w, v \in W$, $w \neq v$, $(w)$ is a convergent, constant gnet and $f_\lambda$ may be an arbitrary gnet (in $W$). Then:

$[(f_\lambda) \rightarrow w \Rightarrow (f_\lambda) \rightarrow v]$ $\Rightarrow$ $[w \in \bigcap \mathcal{F}_{v}]$.
\end{tw}

\begin{proof}
Suppose that whenever $(f_\lambda) \rightarrow w$, also $(f_\lambda) \rightarrow v$. Let us consider the constant gnet $w, w, w, ...$. It converges to $w$ but also to $v$. Hence, $w$ is eventually in every $G \in \mathcal{F}_{v}$. This means that $w \in \bigcap \mathcal{F}_{v}$. 
\end{proof}

In the next theorem we do not need to assume that $(w)$ is convergent. This thesis is just like ($\Leftarrow$) from the aforementioned theorem.

\begin{tw}
Let $\langle W, \mu, \mathcal{F} \rangle$ be a \gtf. Assume that $w, v \in W$, $w \neq v$ and $(f_\lambda)$ is an arbitrary gnet in $W$. Then:

$[w \in \bigcap \mathcal{F}_{v}]$ $\Rightarrow$ $[(f_\lambda) \rightarrow w \Rightarrow (f_\lambda) \rightarrow v]$
\end{tw}

\begin{proof}
Suppose that $w \in \bigcap \mathcal{F}_{v}$. It allows us to say that any $H \in \mathcal{F}_{v}$ is also in $\mathcal{F}_{w}$. Hence, $\mathcal{F}_{v} \subseteq \mathcal{F}_{w}$. Now assume that $(f_\lambda) \rightarrow w$. Thus, $(f_\lambda)$ is eventually in every $G \in \mathcal{F}_{w}$. In particular, it is in every $G \in \mathcal{F}_{v}$. Clearly, this means that $(f_\lambda) \rightarrow v$. 
\end{proof}

The last lemma in this section is an interesting and useful observation. 

\begin{lem}
\label{pusty}
Let $\langle W, \mu, \mathcal{F} \rangle$ be a \gtf and $(f_\lambda)$ be a gnet. If $(f_\lambda) \rightarrow w \in W$, then $\emptyset \notin \mathcal{F}_{w}$. 
\end{lem}

\begin{proof}
Assume that $\emptyset \in \mathcal{F}_{w}$. By convergence, we know that for any $G \in \mathcal{F}_{w}$, so also for $\emptyset$, there is $\lambda_0 \in P$ such that for each $\lambda \geq \lambda_0$, $f_\lambda \in \emptyset$. This is impossible. 
\end{proof}

The last lemma in this section is a modification of Th. 2.6. in \cite{baskar}. 

\begin{lem}
\label{max}
Let $\langle W, \mu, \mathcal{F} \rangle$ be a \gtf and $f: P \rightarrow W$ be a gnet in $W$. Assume that $m$ is a maximal element of $P$ and $f_{m} \in \bigcup \mu$ or $\mathcal{F}_{f_{m}} = \emptyset$. Then $(f_\lambda) \rightarrow f_{m}$. 
\end{lem}

\begin{proof}
If $f_{m} \in \bigcup \mu$, then let us consider an arbitrary $G \in \mathcal{F}_{f_{m}}$. Of course, $f_{m} \in G$ and $f_\lambda \in G$ for each $\lambda \geq m$. The reason is that $\lambda \geq m$ implies $\lambda = m$. We conclude that $f_\lambda \rightarrow f_{m}$. 

If $\mathcal{F}_{f_{m}}$ is empty, then our result is trivial. 
\end{proof}

In the original Baskaran's theorem there was no need to assume anything special about $f_m$. However, it should be clear for the reader that our assumptions are important. Without them it would be easy to imagine the following situation: that we construct $(f_\lambda)$ in such a way that $f_{m}$ is somewhere beyond $\bigcup \mu$ and there is at least one $G \in \mathcal{F}_{f_{m}}$. Then for any $\lambda \in P$ there is $\lambda_0 \geq \lambda$ (namely $m$) such that $f_{\lambda_0} = f_{m} \notin G$. 

\section{The higher level of convergence}

We have already proved that each point of $W$ is contained in certain $\mathcal{F}$-open neighbourhood (if $\mathcal{F}_{w} \neq \emptyset$). This observation leads us to the second understanding of convergence.

\begin{df}
Let $\langle W, \mu, \mathcal{F} \rangle$ be a \gtf-structure. Assume that $w \in W$. We define $\mathcal{E}_{w}$ as the set of all $\mathcal{F}$-open sets to which $w$ belongs. 
\end{df}

As we know from Th. \ref{eachopen}, $\mathcal{E}_{w} = \emptyset \Leftrightarrow \mathcal{F}_{w} = \emptyset$. Let us go back to the \gtf-structure from Ex. \ref{zet}. 

\begin{przy}
\label{zetdwa}
\normalfont{

\item Recall that basically we are working with $\langle \mathbb{Z}, \mu, \mathcal{F} \rangle$ where $\mu = \{ \emptyset, \{1\}, \{1, 3\}, \{1, 3, 5\}, \\ \{1, 3, 5, 7\}, ... \}$, $\mathcal{F}_{x} = \emptyset$ for any $x \in 2 \mathbb{Z}$ (if $x$ is odd, then $\mathcal{F}_{x}$ is just a collection of its open neighbourhoods; this comes from the general suppositions). 

Assume that $m$ is an odd integer. Consider an arbitrary $G \in \mathcal{F}_{m}$. Note that for any $n \in 2 \mathbb{Z}$, $\mathcal{F}_{n} = \emptyset$. For this reason, $\mathcal{F}Int(G) \subseteq G$, hence (by means of Lemma \ref{amu} and the fact that $G \in \mu$) $G$ is $\mathcal{F}$-open. Of course $m \in G$ (because $m \in \bigcup \mu$), so $G \in \mathcal{E}_{m}$.

On the other hand, assume that there is $H \in \mathcal{E}_{m}$ such that $H \notin \mathcal{F}_{m}$. It means that $w \notin H$ (contradiction) or that $H \notin \mu$. If $H \cap [W \setminus \bigcup \mu] \neq \emptyset$ the we have contradiction again: if there is any $n \in H \cap [W \setminus \bigcup \mu]$ and $H$ is $\mathcal{F}$-open, then it means that $n \in \mathcal{F}Int(H)$, so $\mathcal{F}_{n} \neq \emptyset$. This is not possible (because $n$ is even).

Hence, $H = \mathcal{F}Int(H) \subseteq \bigcup \mu$. All the assumptions of Lemma \ref{uopen} are satisfied. Thus, $H \in \mu$. Of course, $m \in H$. It means that $H \in \mathcal{F}_{m}$. Finally, in this case $\mathcal{F}_{m} = \mathcal{E}_{m}$.
}
\end{przy}

Note that the reasoning presented above is in fact general. Hence, we can formulat the following conclusion:

\begin{tw}
Assume that $\langle W, \mu, \mathcal{F} \rangle$ is a \gtf-structure and $\mathcal{F}_{w} = \emptyset$ for each $w \in W \setminus \bigcup \mu$. Then for any $v \in \bigcup \mu$, $\mathcal{F}_{v} = \mathcal{E}_{v}$. Moreover, this result is true also for any $w \in W \setminus \bigcup \mu$: $\mathcal{F}_{w} = \mathcal{E}_{w} = \emptyset$.
\end{tw}

Now we can go further:

\begin{df}
Let $\langle W, \mu, \mathcal{F} \rangle$ be a \gtf and $(f_\lambda)$ be a gnet in $W$. We say that:

\begin{itemize}

\item $(f_\lambda)$ $\mathcal{E}$-converges to $w \in W$ (i.e. $(f_\lambda) \rightarrow^{\mathcal{E}} w$) iff for any $G \in \mathcal{E}_{w}$, $f_\lambda$ is eventually in $G$. In this case we say that $w$ is an $\mathcal{E}$-limit of $(f_\lambda)$. We say that $(f_\lambda)$ is $\mathcal{E}$-convergent if there is $v \in W$ such that $(f_\lambda) \rightarrow^{\mathcal{E}} v$. 

\item We say that $w$ is an $\mathcal{E}$-limit point of $(f_\lambda)$ if it is frequently in every $G \in \mathcal{E}_{w}$. 
\end{itemize}

\end{df}

What are the properties of such convergence? Let us start from constant gnets.

\begin{lem}
Each constant gnet in any \gtf-structure $\langle W, \mu, \mathcal{F} \rangle$ is $\mathcal{E}$-convergent. 
\end{lem}

\begin{proof}
Let us consider $(f_\lambda) = (w)$. Suppose that for any $v \in W$, $(w) \nrightarrow^{\mathcal{E}} v$. Hence, for any $v \in W$ there is $S \in \mathcal{E}_{v}$ such that $w \notin S$. In particular, this is true for $v = w$. Thus, there is $S \in \mathcal{E}_{w}$ such that $w \notin S$. This is impossible because of the very definition of $\mathcal{E}_{w}$.
\end{proof}

\begin{lem}
\label{econv}
Let $\langle W, \mu, \mathcal{F} \rangle$ be a \gtf and $(f_\lambda) = (w)$ be a constant gnet in $W$. Then $(w)$ is $\mathcal{E}$-convergent $\Leftrightarrow$ $(w) \rightarrow^{\mathcal{E}} w$. 
\end{lem}

\begin{proof}
Assume that $(w) \nrightarrow^{\mathcal{E}} w$. It means that there is $S \in \mathcal{E}_{w}$ such that $w \notin S$. This is contradiction. 
\end{proof}

Now we introduce the notion of $\mathcal{E}T_1$-spaces. 

\begin{df}
We say that \gtf-structure $\langle W, \mu, \mathcal{F} \rangle$ is $\mathcal{E}T_1$ $\Leftrightarrow$ for any $w \neq v$ there are $G \in \mathcal{E}_{w}$ such that $v \notin G$ and $H \in \mathcal{E}_{v}$ such that $w \notin H$. 
\end{df}

We can prove the following theorem about uniqueness:

\begin{tw}
Let $\langle W, \mu, \mathcal{F} \rangle$ be a \gtf-structure. Then the $\mathcal{E}$-limit of every constant gnet is unique $\Leftrightarrow$ $\langle W, \mu, \mathcal{F} \rangle$ is $\mathcal{E}T_1$. 
\end{tw}

\begin{proof}
($\Rightarrow$)
Suppose that $(w)$ is $\mathcal{E}$-convergent. We may assume that $(w) \rightarrow^{\mathcal{E}} w$. For any $v \neq w$, $(w) \nrightarrow^{\mathcal{E}} v$, i.e. there is $H \in \mathcal{E}_{v}$ such that $w \notin H$.

But maybe for any $S \in \mathcal{E}_{w}$, $v \in S$? Let us consider constant gnet $(v)$. Then $(v) \rightarrow^{\mathcal{E}} v$. But then $(v) \nrightarrow^{\mathcal{E}} w$. Hence, there must be $G \in \mathcal{E}_{w}$ such that $v \notin G$. 

($\Leftarrow$)
Assume that there is constant gnet $(w)$ with two different $\mathcal{E}$-limits, i.e. $(w) \rightarrow^{\mathcal{E}} w$ and $(w) \rightarrow^{\mathcal{E}} v \neq w$. It means that for any $S \in \mathcal{E}_{v}$, $w \in S$. Contradiction.
\end{proof}

Below we prove certain connection between convergence and $\mathcal{E}$-convergence.

\begin{tw}
\label{conv}
Let $\langle W, \mu, \mathcal{F} \rangle$ be a \gtf and $(f_\lambda)$ be a gnet (we assume that $f: P \rightarrow W$). If $(f_\lambda) \rightarrow w$, then $(f_\lambda) \rightarrow^{\mathcal{E}} w$.
\end{tw}

\begin{proof}
Suppose that $(f_\lambda) \nrightarrow^{\mathcal{E}} w$. Then there is $S \in \mathcal{E}_{w}$ such that for any $\lambda \in P$ there exists $\lambda_1 \geq \lambda$ for which $f_{\lambda_1} \notin S$. 

We know that $S \neq \emptyset$ (because $S \in \mathcal{E}_{w}$, so $w \in S$). Moreover, $S$ is $\mathcal{F}$-o., so $w \in \mathcal{F}Int(S)$. Hence, there is $H \in \mathcal{F}_{w}$ such that $H \subseteq S$. Recall that $(f_\lambda) \rightarrow w$, so there is $\lambda_0 \in P$ such that for any $\lambda \geq \lambda_0$, $f_\lambda \in H \subseteq S$. Contradiction.

\end{proof}

Let us go back to the \gtf-structure from Ex. \ref{zet}. Our considerations can be compared with Ex. 3.1. in \cite{baskar}. 

\begin{przy}
\label{zetdwa}
\normalfont{

As earlier, we are working with $\langle \mathbb{Z}, \mu, \mathcal{F} \rangle$ where $\mu = \{ \emptyset, \{1\}, \{1, 3\}, \{1, 3, 5\}, \{1, 3, 5, 7\}, ... \}$ and $\mathcal{F}_{n} = \emptyset$ for any $n \in 2 \mathbb{Z}$. 

Assume that $(P, \leq)$ is a poset, where $P = 2^{\mathbb{Z}} \setminus \{\emptyset\}$ and if $A, B \in P$, then $A \leq B \Leftrightarrow B \subseteq A$. Let us define $f: P \rightarrow \mathbb{Z}$ by $f(A) \in A$ (i.e. we require only this condition). Such $f$ is a gnet. 

Let us discuss $\rightarrow$-convergence and $\mathcal{E}$-convergence. Assume that $m$ is an odd integer. Then $\mathcal{F}_{m} = \{ \{1, 3, ..., m\}, \{1, 3, ..., m, m+2 \}, \{1, 3, ..., m, m+2, m+4\}, ...\}$. If $G \in \mathcal{F}_{m}$, then for every $A \subseteq G$, $f(A) \in A \subseteq G$. In particular, it means that $(f_\lambda) \rightarrow m$ ($G$ itself plays the role of $\lambda_0$ in the general definition). Hence, we can say that every odd integer $m$ is a limit of $f$. Note that $m$ is \emph{not} a limit point of $f$, because for any $G \in \mathcal{F}_{m}$ we can take $B = \{2, 4, 6, ..., m-1\}$. Now $f(A) \notin G$ for any $A \subseteq G$. 

Now consider $n \in 2 \mathbb{Z}$. We have assumed that in this case $\mathcal{F}_{n} = \emptyset$, hence we can immediately say that $(f_\lambda) \rightarrow n$. We see that $(f_\lambda)$ converges to every even integer. 

Now let us think about $\mathcal{G}$, which can be just like $\mathcal{F}'$, $\mathcal{F}''$ or $\mathcal{F}'''$ in Ex. \ref{zet}. The case of odd numbers is without changes. As for the $n \in 2 \mathbb{Z}$, assume that $G \in \mathcal{G}_{n}$ and $A \subseteq G$. Then $f(A) \in A \subseteq G$. In fact, this reasoning is identical with the one for odd integers.

By means of Th. \ref{conv} we can say that in each case, both for odd and even integers, $(f_\lambda)$ is $\mathcal{E}$-convergent to each number. 

}
\end{przy}

We can easily prove that the converse of Th. \ref{conv} is not true. 

\begin{przy}
\normalfont{
Let us consider very simple \gtf-structure: $\langle W, \mu, \mathcal{F} \rangle$, where $W = \{w, v\}$, $\mu = \{ \emptyset, \{w\} \}$ and $\mathcal{F}_{v} = \{ \{w\} \}$. Then the set $\{w, v\} = W$ is $\mathcal{F}$-open and it is the only element of $\mathcal{E}_{v}$ (note that $\mathcal{F}Int(\{v\}) = \emptyset$). Now let us think about constant gnet $(f_\lambda) = (v)$ (connected to an arbitrary $P$). Undoubtedly, $(v) \rightarrow^{\mathcal{E}} v$. Note, however, that $(v) \nrightarrow v$ because there is $G \in \mathcal{F}_{v}$, namely $\{w\}$ such that $v \notin G$. 
}
\end{przy}

We can also reformulate Lemma \ref{max}. Now we do not need any special assumptions about $f_m$.

\begin{lem}
Let $\langle W, \mu, \mathcal{F} \rangle$ be a \gtf and $f: P \rightarrow W$ be a gnet in $W$. Assume that $m$ is a maximal element of $P$. Then $(f_\lambda) \rightarrow^{\mathcal{E}} f_{m}$. 
\end{lem}

\begin{proof}
If $\mathcal{E}_{f_{m}} = \emptyset$, then the result is trivial. If not, then consider $S \in \mathcal{E}_{f_{m}}$. Clearly, $f_{m} \in S$ and $f_\lambda \in S$ for any $\lambda \geq m$ (because in such a case $\lambda = m$). 
\end{proof}

\section{Gnets and the question of closure}

There is a strict dependence between closures and gnets in generalized topology. It has been proven by Baskaran in \cite{baskar} that if $\emptyset \neq A \subseteq W$ and  $w \in W$, then $w \in Cl(A)$ $\Leftrightarrow$ there is a gnet $(f_\lambda)$ in $A$ (i.e. with its values in $A$) converging to $w$. However, Baskaran considered the notion of convergence based on typical open neighbourhoods. Hence, he assumed that each point is in each of its neighbourhoods. Clearly, in our case this is false (for all points from $W \setminus \bigcup \mu$). For this reason, we formulate the following dependece:

\begin{tw}
Assume that $\langle W, \mu, \mathcal{F} \rangle$ is a \gtf-structure, $\emptyset \neq A \subseteq W$. Then:

$w \in \mathcal{F}Cl(A)$ $\Leftrightarrow$ there is a gnet $(f_\lambda) \in A$ such that $(f_\lambda) \rightarrow w$. 
\end{tw} 

\begin{proof}
($\Rightarrow$) Assume that $w \in \mathcal{F}Cl(A)$. There are two possibilites. First, $\mathcal{F}_{w} = \emptyset$. In this case we may assume that $P = 2^{W} \setminus \{ \emptyset \}$ and $C \geq D \Leftrightarrow C \subseteq D$. We define $f: P \rightarrow W$ in such a way that $f(C) \in A$. Clearly, $(f_\lambda)$ becomes a gnet in $A$ and moreover $(f_\lambda) \rightarrow w$ (because there are no any sets in $\mathcal{F}_{w}$, so we can say anything about them). 

Second option is that $\mathcal{F}_{w} \neq \emptyset$. Here we assume that $P = \mathcal{F}_{w}$. As for the $\geq$, it is defined as above. Note that (from the very definition of $\mathcal{F}Cl(A)$) for any $G \in \mathcal{F}_{w}$, $G \cap A \neq \emptyset$. Then assume \footnote{This reasoning is based on the one for ordinary generalized neighbourhoods, presented in \cite{baskar}. However, there is a mistake there (probably a typo). The author assumed only that $f(G) \in G$ (we use our notation). Clearly, we must assume that our gnet is in the (nonempty) intersection of neighbourhood and the set $A$.} that $f(G) \in G \cap A$ for any $G \in \mathcal{F}_{w}$. Then $(f_\lambda)$ is a gnet in $A$ and for any $G \in \mathcal{F}_{w}$ our gnet is eventually in $G$, i.e. $(f_\lambda) \rightarrow w$. 

($\Leftarrow$) Assume that there is a gnet $(f_\lambda)$ in $A$ such that $(f_\lambda) \rightarrow w$. Hence, for any $G \in \mathcal{F}_{w}$, $(f_\lambda)$ is eventually in $G$, which means that for any $G \in \mathcal{F}_{w}$ there is $\lambda_0$ such that for any $\lambda \geq \lambda_0$, $f_\lambda \in G$. But for any $\lambda$, $f_\lambda \in A$. Hence, for any $G \in \mathcal{F}_{w}$, $G \cap A \neq \emptyset$. Thus $w \in \mathcal{F}Cl(A)$. Moreover, due to the Lemma \ref{close}, $w \in Cl(A)$. 

\end{proof}

Is it possible to replace $\rightarrow$-convergence by $\rightarrow^{\mathcal{E}}$-convergence? Of course, if $w \in \mathcal{F}Cl(A)$, then we can find our expected $\rightarrow$-convergent gnet, as it has been shown above: and this gnet is (by means of Lemma \ref{econv}) $\rightarrow^{\mathcal{E}}$-convergent. But the converse is not true. Let us think about the following (counter)-example.

\begin{przy}
\normalfont{
Let $\langle W, \mu, \mathcal{F} \rangle$ be a \gtf-structure where $W = \{w, v, u\}$, $\mu = \{\emptyset, \{w\}\}$, $\mathcal{F}_{v} = \{ \{w\}, \{u\} \}$, $\mathcal{F}_{u} = \emptyset$. Of course, $\mathcal{F}_{w} = \{ \{w \} \}$. Consider $A = \{v \}$ and the constant gnet $(v)$. Clearly, $(v)$ is a gnet in $A$. It is $\rightarrow^{\mathcal{E}}$-convergent (at least to $v$). 

Now $v \notin \mathcal{F}Cl(A)$ because there is $G = \{u\} \in \mathcal{F}_{v}$ such that $G \cap \{v\} = G \cap A = \{u\} \cap \{v\} = \emptyset$. In fact, $\mathcal{F}Cl(A) = \mathcal{F}Cl(\{v\}) = \{z \in W; \text{ for any } G \in \mathcal{F}_{z}, G \cap \{v\} \neq \emptyset\} = \{u\}$ (because there are no any sets in $\mathcal{F}_{u}$).
}
\end{przy}

We have introduced $\mathcal{F}$-interiors (closures) to speak later about $\mathcal{F}$-open (closed) sets. Then we have discussed the set $\mathcal{E}_{w}$ for an arbitrary $w$. One could ask: does it make sense to move these notions on even more high level? We wish to treat this issue very briefly.

\begin{df}
Assume that $\langle W, \mu, \mathcal{F} \rangle$ is a \gtf-structure and $A \subseteq W$. We say that:

\begin{itemize}
\item $w \in \mathcal{E}Int(A)$ $\Leftrightarrow$ there is $S \in \mathcal{E}_{w}$ such that $S \subseteq A$.
\item $w \in \mathcal{E}Cl(A)$ $\Leftrightarrow$ for any $S \in \mathcal{E}_{w}$, $S \cap A \neq \emptyset$.
\end{itemize}

We say that $A$ is $\mathcal{E}$-open if and only if $\mathcal{E}Int(A) = A$ and that $A$ is $\mathcal{E}$-closed if and only if $\mathcal{E}Cl(A) = A$. 
\end{df}

Although on this stage of research such definitions seem to be somewhat artificial, there is at least one interesting thing to note. 

\begin{tw}
Let $\langle W, \mu, \mathcal{F} \rangle$ be a \gtf-structure and $\mathcal{EO}$ is a collection of all $\mathcal{E}$-open sets (with respect to $\mu$ and $\mathcal{F}$). Then $\mathcal{EO}$ forms a generalized topology on $W$. If for any $w \in W$, $\mathcal{F}_{w} \neq \emptyset$, then $\mathcal{EO}$ is a strong generalized topology.
\end{tw}

\begin{proof}
First, let us prove that $\emptyset$ is $\mathcal{E}$-open. Compute: $\mathcal{E}Int(\emptyset) = \{z \in W; \text{ there is } S \in \mathcal{E}_{w} \text { such that } S \subseteq \emptyset\} = \emptyset$. It is because the only set contained in $\emptyset$ is empty set itself - but for any $S \in \mathcal{E}_{w}$, $w \in S$, hence $S \neq \emptyset$. 

Assume now that $J \neq \emptyset$ and for any $i \in J$, $X_{i}$ is $\mathcal{E}$-open. Then $\bigcup X_{i}$ is also $\mathcal{E}$-open, i.e. $\mathcal{E}Int(\bigcup_{i \in J} X_i) = \bigcup X_{i}$. 

($\subseteq$) Let $w \in \mathcal{E}Int(\bigcup X_i)$. Hence, there is $S \in \mathcal{E}_{w}$ such that $S \subseteq \bigcup_{i \in J} X_i$. But $w \in S$. Hence $w \in \bigcup_{i \in J} X_i$. 

($\supseteq$) Let $w \in \bigcup_{i \in J} X_i$. Then there is $X_k$ such that $w \in X_k$. $X_k$ is $\mathcal{E}$-open, so $w \in \mathcal{E}Int(X_k)$. Thus there is $S \in \mathcal{E}_{w}$ such that $S \subseteq X_k \subseteq \bigcup_{i \in J} X_i$. Hence, $w \in \mathcal{E}Int(\bigcup_{i \in J}X_i)$. 

Now assume that $\mathcal{F}_{w} \neq \emptyset$ for any $w \in W$. Then $\mathcal{E}_{w} \neq \emptyset$. Hence, $\mathcal{E}Int(W) = \{z \in W; \text{ there is } S \in \mathcal{E}_{w}  \text{ such that } S \subseteq W\} = W$. 
\end{proof}

Moreover, it is always true that $\mathcal{E}Int(\bigcap_{i \in J} X_i \subseteq \bigcap X_i$ (even if none of the sets indexed by $J$ is $\mathcal{E}$-open). Assume that $w \in \mathcal{E}Int(\bigcap_{i \in J}X_i$. Hence, there is $S \in \mathcal{E}_{w}$ such that $S \subseteq \bigcap_{i \in J} X_i$. Hence, for any $X_i$ we can say that $S \subseteq X_i$. But by the very definition of $\mathcal{E}_{w}$, $w \in S$. So $w$ is in each $X_i$, i.e. it is in $\bigcap_{i \in J} X_i$. 

\section{Further investigations}

We would like to investigate the structures and functions presented above. It would be cognitively valuable to establish analogues of various topological notions (like density, nowhere density, connectedness, continuity etc.) both in the context of $\mu$-topology with $\mathcal{F}$ and in the context of $\mathcal{E}$ (i.e. $\mathcal{F}$-open and $\mathcal{E}$-open sets). It would be very natural to impose certain restrictions on $\mathcal{F}$ (that, for example, it is closed under supersets or unions). Logical applications of our new notions are also interesting: because in the presence of $\mathcal{F}$ we can discuss two types of possible worlds; those in $\bigcup \mu$ and those in $W \setminus \bigcup \mu$. Their logical "strenght" (in terms of validating given formulas and rules) is different.


\begin{thebibliography}{90}

\bibitem{baskar}R. Baskaran, \emph{Generalized nets in generalized topological spaces}, Journal of Advanced Research in Pure Mathematics, 2011, pp. 1-7. 

\bibitem{seq}R. Baskaran, \emph{Sequential Convergence in Generalized Topological Spaces}, Journal of Advanced Research in Pure Mathematics, Vol. 3, Issue 1, 2011. 

\bibitem{csaszar} \'{A}. Cs\'{a}sz\'{a}r, \emph{Generalized topology, generalized continuity}, Acta Mathematica Hungarica, 96(4) (2002), 351 - 357.

\bibitem{geno} \'{A}. Cs\'{a}sz\'{a}r, \emph{Generalized open sets in generalized topologies}, Acta Mathematica Hungarica, 106 (2005), 53-66.

\bibitem{genopen} \'{A}. Cs\'{a}sz\'{a}r, \emph{Generalized open sets}, Acta Mathematica Hungarica, 75 (1997), 65-87.

\bibitem{sep} \'{A}. Cs\'{a}sz\'{a}r, \emph{Separation axioms for generalized topologies}, Acta Mathematica Hungarica, 104 (1-2) (2004), 63 - 69. 

\bibitem{groups} M. Hussain, M. ud Din Khan, C. \"{O}zel, \emph{On generalized topological groups}, Filomat, no. 27 (4), May 2012. 

\bibitem{indrze} A. Indrzejczak, \emph{Labelled tableau calculi for weak modal logics}, Bulletin of the Section of Logic, Volume 36: 3/4 (2007), pp. 159 - 171. 

\bibitem{jarvi} J. J\"{a}rvinen, M. Kondo, J. Kortelainen, \emph{Logics from Galois connections}, International Journal of Approximate Reasoning, vol. 49, Issue 3, November 2008, pages 595 - 606. 

\bibitem{pacuit} E. Pacuit, \emph{neighbourhood Semantics for Modal Logic}, Springer International Publishing AG 2017.

\bibitem{filters} S. Palaniammal, M. Murugalingam, \emph{Generalized filters}, International Mathematical Forum, Vol. 9, 2014, no. 36.

\bibitem{separ} A. P. Dhana Balan, P. Padma, \emph{Separation spaces in generalized topology}, International Journal of Mathematics Research, Volume 9, Number 1 (2017), pp. 65 - 74.

\bibitem{sarma}R. D. Sarma, \emph{On convergence in generalized topology}, International Journal of Pure and Applied Mathematics 60(2), January 2010.

\bibitem{sarsak}M. S. Sarsak, \emph{New separation axioms in generalized topological spaces}, Acta Mathematica Hungarica, 132 (3) (2011), 244 - 252.

\bibitem{soldano} H. Soldano, \emph{A modal view on abstract learning and reasoning}, 
Ninth Symposium on Abstraction, Reformulation, and Approximation, SARA 2011.

\bibitem{witczak} T. Witczak, \emph{Generalized topological semantics for weak modal logics}, \url{https://arxiv.org/pdf/1904.06099.pdf}

\end{thebibliography}
\end{document}